\documentclass[12pt]{article}
\usepackage[utf8]{inputenc}
\usepackage{amsthm}
\usepackage[margin=1.5in]{geometry}
\usepackage{todonotes}

\usepackage{hyperref, xcolor}
\usepackage{scrextend}
\hypersetup{
    colorlinks=true,
    urlcolor=magenta,
}
\usepackage{academicons}
\definecolor{orcidlogocol}{HTML}{A6CE39}
\usepackage{amsmath,amssymb,amsfonts,oldlfont,graphics,oldgerm,latexsym,dsfont}
\usepackage{graphics}
\usepackage{pstricks,pst-node,pst-arrow}
\usepackage{latexsym,amssymb}
\foreach \x in {A, ..., Z}{%
	\expandafter\xdef\csname orcid\x\endcsname{\noexpand\href{https://orcid.org/\csname orcidauthor\x\endcsname}{\noexpand\orcidicon}}
}

\usepackage{enumerate}
\usepackage{enumitem}

\PassOptionsToPackage{usenames,dvipsnames,svgnames}{xcolor}
\usepackage{float}

\newtheorem{theorem}{Theorem}

\newtheorem{proposition}[theorem]{Proposition}
\newtheorem{observation}[theorem]{Observation}
\newtheorem{conjecture}[theorem]{Conjecture}
\usepackage{mathtools}
\DeclarePairedDelimiter{\ceil}{\lceil}{\rceil}
\DeclarePairedDelimiter\floor{\lfloor}{\rfloor}
\usepackage{amssymb}

\newcommand{\fsG}{\overleftrightarrow{G}}

\usepackage{graphics}

\providecommand{\keywords}
{
  \small	
  \noindent \textbf{Keywords:} symmetry breaking; symmetric digraphs
}

\providecommand{\msc}
{
  \small	
  \noindent \textbf{Mathematics Subject Classification:} 05C15, 05C20, 05C25
}

\title{Distinguishing symmetric digraphs by proper arc-colourings of type I}
\author{Rafał Kalinowski, Monika Pilśniak, Magdalena Prorok\\
\small AGH University of Krakow\\
\small al. Mickiewicza 30, 30-059 Krakow, Poland\\
\small\tt {kalinows, pilsniak, prorok}@agh.edu.pl}

\begin{document}

\maketitle

\begin{abstract}
A symmetric digraph $\overleftrightarrow{G}$ is obtained from a simple graph $G$ by replacing each edge $uv$ with a pair of opposite arcs $\overrightarrow{uv}$, $\overrightarrow{vu}$. An arc-colouring $c$ of a digraph $\overleftrightarrow{G}$ is distinguishing if the only automorphism of $\overleftrightarrow{G}$ preserving the colouring $c$ is the identity. Behzad introduced the proper arc-colouring of type I as an arc-colouring such that any two consecutive arcs $\overrightarrow{uv}$, $\overrightarrow{vw}$ have distinct colours. We establish an optimal upper bound $\lceil 2\sqrt{\Delta(G)}\rceil$ for the least number of colours in a distinguishing proper colouring of type I of a connected symmetric digraph $\overleftrightarrow{G}$. Furthermore, we prove that the same upper bound $\lceil 2\sqrt{\Delta(G)}\rceil$ is optimal for another type of proper colouring of $\overleftrightarrow{G}$, when only monochromatic 2-paths are forbidden. 
\end{abstract}

\keywords

\msc

\section{Introduction}
We use standard graph theory terminology and notation.
An edge-colouring $c$ of a~graph $G$ is called {\it distinguishing} if the identity is the only automorphism preserving~$c$. 
In 2015, Kalinowski and Pilśniak~\cite{KP} introduced the {\it distinguishing chromatic index} $\chi'_D(G)$ of a graph $G$ as the least number of colours in  a proper distinguishing edge-colouring of $G$. In particular, they proved that $\Delta(G) \leq \chi'_D(G) \leq \Delta(G)+1$ for every connected graph $G$ of order $|G|\geq 3$ except for four graphs of small order $K_4$, $C_4$, $C_6$, $K_{3,3}$.

By $\overleftrightarrow{G}$ we denote a {\it symmetric digraph} obtained from a simple graph $G$ by replacing each edge $uv$ by a pair of opposite arcs $\overrightarrow{uv}$ and $\overrightarrow{vu}$. 
The concept of distinguishing edge-colourings of graphs can be naturally extended to arc-colouring of digraphs. This problem is particularly interesting for symmetric digraphs since the automorphism group of a symmetric digraph $\overleftrightarrow{G}$ coincides with the automorphism group of the underlying graph $G$. Note that, in general, the automorphism group of a digraph is a subgroup of the automorphism group of the underlying graph. Let us note that distinguishing arc-colourings of digraphs were studied in a different context by Meslem and Sopena~\cite{MS}.

A definition of proper arc-colouring of a digraph depends on a definition of adjacency of arcs. There are 15 possible definitions of a proper arc-colouring of a~digraph since there are 15 possible definitions of adjacency of two arcs corresponding to non-empty forbidden monochromatic subsets of the set of the four digraphs: $2$-cycle $A_1$, $2$-path $A_2$, source $A_3$ and sink $A_4$ (see Figure~\ref{4Ai}).

\vspace{7mm}


\begin{figure}[htb]\label{4Ai}
\centering
\includegraphics[width=8cm]{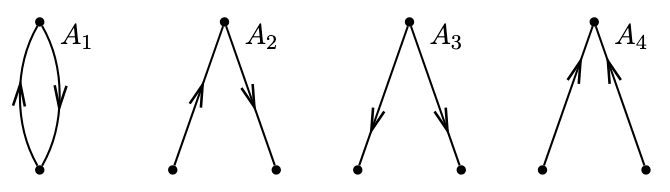}
\caption{Four weakly connected digraphs with two arcs}
\end{figure}


In this paper, we investigate distinguishing proper arc-colouring for two of these definitions, all other 13 ones have already been investigated in \cite{digrafyI} and \cite{KPP}. One of these two definitions has already been studied in the literature. An arc-colouring of a~digraph is called {\it proper of type I} if there are neither monochromatic 2-cycles nor monochromatic 2-paths. We denote by $\chi'_{1,2}(\overleftrightarrow{G})$ the  chromatic index of type~I of a~symmetric digraph $\overleftrightarrow{G}$, i.e. the least number of colours in a proper arc-colouring of $\overleftrightarrow{G}$ of type I. We also use the notation $\chi'_{2}(\overleftrightarrow{G})$ for another chromatic index of $\fsG$ when only monochromatic $2$-paths are forbidden. 

A proper colouring of type I in digraphs was introduced by Behzad~\cite{Behzad} in 1965, and then investigated by Harner and Entringer~\cite{HE}. Poljak and R\"odl~\cite{PR} proved the following notable result in 1981. 
\begin{theorem} {\rm (\cite{PR})}\label{PRodl}
For every graph $G$
$$\chi'_{1,2}(\overleftrightarrow{G})=\min\left\{k\;:\;\chi(G)\leq{k\choose \lfloor k/2\rfloor}\right\},$$
where $\chi(G)$ is the chromatic number of the underlying graph $G$.
\end{theorem}

Analogously, $\chi'_{D_{1,2}}(\overleftrightarrow{G})$ and $\chi'_{D_{2}}(\overleftrightarrow{G})$  stand for the distinguishing chromatic indices of~$\overleftrightarrow{G}$, i.e. the least number of colours in a~distinguishing proper arc-colouring, where the indicated two-arc digraphs cannot be monochromatic.

\vspace{3mm}
In Section~\ref{sec:S2}, we discuss distinguishing proper colouring of type I of symmetric digraphs. 
There, we formulate our main result, Theorem~\ref{thm:main}, which provides an~optimal upper bound for the distinguishing chromatic index $\chi'_{D_{1,2}}(\overleftrightarrow{G})$. Its proof is given in Section \ref{proof}. 
Finally in Section \ref{Slast}, we discuss proper arc-colourings, where only monochromatic 2-paths are forbidden.

We restrict our investigations of distinguishing colouring to connected graphs $G$ to avoid dealing with isomorphic components.

Given an arc-colouring of $\fsG$, we write that a vertex $v$ is {\it fixed}, if $v$ is a fixed point of every automorphism of $\fsG$ preserving this colouring.


\section{Distinguishing chromatic index \texorpdfstring{$\chi'_{D_{1,2}}(\fsG)$}{Lg}}\label{sec:S2}

In this section, we investigate the distinguishing chromatic index $\chi'_{D_{1,2}}(\fsG)$. Clearly, for every connected symmetric digraph $\overleftrightarrow{G}$ we have $ \chi'_{D_{1,2}} (\overleftrightarrow{G})\geq \chi'_{1,2} (\overleftrightarrow{G})$. First, we show that the equality is achieved for complete graphs.

\begin{proposition}\label{chi12}
Every proper arc-colouring of type I of a complete symmetric digraph $\overleftrightarrow{K_n}$ is distinguishing. Hence, $$\chi'_{D_{1,2}}(\overleftrightarrow{K_n}) = \chi'_{1,2}(\overleftrightarrow{K_n})  = \min \left\{k: n\leq \binom{k}{\floor{\frac{k}{2}}} \right\}.$$
\end{proposition}

\begin{proof}
Let $c$ be any proper arc-colouring of type I of $\overleftrightarrow{K_n}$. Suppose $c$ is not distinguishing. Therefore, there exists a non-trivial automorphism  $\varphi$ preserving the colouring $c$. Let $u$ and $v$ be two distinct vertices of $\overleftrightarrow{K_n}$ such that $\varphi(u) = v$. If the arc $\overrightarrow{uv}$ is coloured with $\alpha$, then there must exist an arc $\overrightarrow{vw}$ of colour $\alpha$ outgoing from the vertex $v$. However, this would create a monochromatic $2$-path if $w\neq u$, or a monochromatic 2-cycle if $w=u$, a contradiction.
\end{proof}

Another example of a graph $G$ that achieves the equality $\chi'_{D_{1,2}} (\overleftrightarrow{G})=\chi'_{1,2}(\overleftrightarrow{G})$ is an odd cycle. We now determine the distinguishing chromatic index $\chi'_{D_{1,2}}$ of any path and cycle, which will be useful in the proof of the main theorem.

\begin{observation}\label{cycles}
For symmetric directed paths $\overleftrightarrow{P_n}$ we have $\chi'_{D_{1,2}}(\overleftrightarrow{P_{2k}}) = 2$ and $\chi'_{D_{1,2}}(\overleftrightarrow{P_{2k+1}}) = 3$.
For every symmetric directed cycle $\overleftrightarrow{C_n}$ with $n \geq 3$, we have $\chi'_{D_{1,2}}(\overleftrightarrow{C_n}) = 3$.
\end{observation}

\begin{proof}
For a path $\overleftrightarrow{P_n}$, there is a proper arc-colouring of type I with two colours, which is unique up to the swapping of colours. The only non-trivial automorphism of $\overleftrightarrow{P_{2k}}$, a reflection, does not preserve this colouring since $c(\overrightarrow{v_1v_2}) \neq c(\overrightarrow{v_nv_{n-1}})$.
For $\overleftrightarrow{P_{2k+1}}$, we need a third colour for one arc, say $\overrightarrow{v_1v_2}$. The vertex $v_1$ is thus fixed, and consequently, all vertices are fixed.

For even cycles $\overleftrightarrow{C_{2k}}$, there is exactly one proper colouring of type I with two colours up to the swapping of colours. This colouring is not distinguishing. We change the colour of the arc $\overrightarrow{v_1v_2}$ to a~third colour. The vertices $v_1$ and $v_2$ are thus fixed and, consequently, all vertices are fixed, too.

For odd cycles, we have $\chi'_{1,2}(\overleftrightarrow{C_{2k+1}})=3$, and there exists a proper colouring $c$ with three colours, where only two arcs $\overrightarrow{v_nv_1}$, $\overrightarrow{v_2v_1}$ have a third colour. The vertices $v_1$, $v_2$ and $v_n$ are fixed, because they have unique sets of colours of ingoing arcs. Consequently, the colouring is distinguishing.
\end{proof}

The main result of this paper is the following optimal upper bound for the distinguishing chromatic index of type I for any connected symmetric digraph.

\begin{theorem} \label{thm:main}
If $G$ is a connected graph of maximum degree $\Delta$, then $$\chi'_{D_{1,2}}(\overleftrightarrow{G}) \leq \ceil[\Big]{2 \sqrt{\Delta}}.$$
\end{theorem}

The proof of Theorem {\ref{thm:main}} is given in Section {\ref {proof}}.  
Let us now observe that the equality holds for symmetric stars $\overleftrightarrow{K_{1, \Delta}}$, while only two colours are enough in a~proper arc-colouring, i.e. $\chi'_{{1,2}}(\overleftrightarrow{K_{1, \Delta}})=2$ by Theorem \ref{PRodl}. This shows that the difference between the chromatic index $\chi'_{{1,2}}(\overleftrightarrow{G})$ and the distinguishing chromatic index $\chi'_{D_{1,2}}(\overleftrightarrow{G})$ can be arbitrarily large.

\begin{proposition}\label{stars}
For the symmetric directed star $\overleftrightarrow{K_{1,\Delta}}$ we have $$\chi'_{D_{1,2}}(\overleftrightarrow{K_{1,\Delta}}) = \ceil[\Big]{2 \sqrt{\Delta}}. $$
\end{proposition}

\begin{proof}
Set $k= \ceil[\Big]{2 \sqrt{\Delta}}.$  Let $w$ be a vertex of degree $\Delta$ in $K_{1,\Delta}$.
In a  proper distinguishing colouring, every 2-cycle has to be coloured with a different pair of two colours. Moreover, if a colour $\alpha$ is used on an arc $\overrightarrow{uw}$, then it cannot be used on any arc $\overrightarrow{wv}$, hence the sets of colours for arcs ingoing to $w$  and those outgoing  from $w$ are disjoint. We need $\Delta$ pairs of colours. We use colours $\Big\{1, 2, \ldots, \ceil[\big]{\frac{k}{2}}  \Big\}$ for the arcs from $w$ to $\Delta$ pendant vertices and $\Big\{ \ceil[\big]{\frac{k}{2}}+1, \ceil[\big]{\frac{k}{2}}+2, \ldots, k \Big\}$ colours for the arcs in the opposite direction. It is easy to see that $$\ceil[\bigg]{\frac{k}{2}} \cdot \floor[\bigg]{\frac{k}{2}} \geq \Delta.$$ 
\end{proof}


\section{Proof of Theorem \ref{thm:main}}\label{proof}

If $\Delta(G)=1$, then $G=K_2$, and trivially, $\chi'_{D_{1,2}}(\overleftrightarrow{K_2})=2$. For $\Delta=2$ the claim follows by Observation~\ref{cycles}. Now, we prove Theorem \ref{thm:main} for $\Delta \geq 3$. We will use the terms {\it parent, child} and {\it sibling} with respect to a given BFS tree.

Let $G$ be a connected graph with maximum degree $\Delta\geq 3$. We construct a~distinguishing proper arc-colouring $c$ of the symmetric digraph $\fsG$ using colours from the set $\{1,\ldots,k\}$, where $k=\ceil[\Big]{2 \sqrt{\Delta}}$.

To ensure that $c$ is proper, for each vertex $v$ we assign a~list $L(v)$ of $\lfloor \frac k2\rfloor$ admissible colours for arcs ingoing to $v$, such that adjacent vertices get distinct lists. Then the arcs of $\fsG$ are coloured according to the rule
$${\rm { }} \qquad \qquad \qquad\qquad \qquad \qquad c(\overrightarrow{uv})\in L(v)\setminus L(u)\qquad \qquad \qquad\qquad \qquad \qquad \qquad (\star) $$
for every arc $\overrightarrow{uv}$ of $\fsG$. Consequently, no 2-cycle or 2-path will be monochromatic. 

To ensure that $c$ is distinguishing, we first pick a vertex $w$ of degree $\Delta$, and define its list $L(w)= \bigl\{ \left\lceil \frac k2\right\rceil+1,\ldots,k \bigr\}$. 
The arcs incident to $w$ induce a star, and we colour them as in the proof of Proposition~\ref{stars}. That is, the arcs ingoing to $w$ get colours from $L(w)$, while the arcs outgoing from $w$ get colours from $\bigl\{1,\ldots,\left\lceil \frac k2\right\rceil \bigr\}$ in such a~way that each neighbour $u$ of $w$ gets a~distinct pair of two colours of the arcs $\overrightarrow{uw},\overrightarrow{wu}$. The vertex $w$ will be the only vertex of $\fsG$ with this colouring of incident arcs, so $w$ will be fixed by every automorphism of $\fsG$ preserving our colouring~$c$.

For every neighbour $u$ of $w$, we assign the~list $L(u)=L(w)\cup\{c(\overrightarrow{wu})\}\setminus\{c(\overrightarrow{uw})\}$, which is distinct from the lists of its siblings, and the colours of arcs between $u$ and $w$ satisfy rule $(\star)$. Observe also that all neighbours of $w$ are fixed, since $w$ is fixed.

We further proceed by considering consecutive vertices of $G$ in a BFS ordering rooted at $w$. At each stage of the procedure, every vertex $x$ with an assigned list $L(x)$ is fixed, and all arcs between vertices with assigned lists are coloured according to rule $(\star)$. Each stage begins with finding the first vertex $v$ in the BFS order, which has a child without an assigned list of admissible colours. Let $A$ be a maximal set of children of $v$ that do not have assigned lists, have the same set of neighbours with already assigned lists, have the same degree, and the same set of neighbours with already assigned lists. Recall that all vertices with assigned lists are already fixed. Consequently, an automorphism $\varphi$ of $\fsG$ can move a child of $v$ only onto a child of $v$ within the same set $A$. 

For each such set $A$, we proceed in three steps: 
\begin{enumerate}[itemsep=-12pt, topsep=4pt]
\item We colour the arcs between the vertex $v$ and the set $A$ so that each vertex of $A$ is fixed. \\
\item We assign a list $L(u)$ to every vertex of $A$ which is distinct from $L(w)$ and from the lists of neighbours of $u$. \\
\item According to rule $(\star)$, we colour all uncoloured arcs between vertices of $A$ and vertices with assigned lists of admissible colours.
\end{enumerate}
After completing this procedure for every set $A$ and every vertex $v$ in the BFS ordering, we clearly obtain a proper distinguishing arc-colouring of $\fsG$.

\vspace{4mm}
The arcs between the vertices of $A$ and the vertex $v$ induce a star. We colour them as in the proof of Proposition~\ref{stars}, that is, each pair $(\overrightarrow{vu},\overrightarrow{uv})$ of opposite arcs between $v$ and $u\in A$ gets a distinct pair of colours $(\alpha,\beta)$ such that $\alpha\notin L(v)$ and $\beta\in L(v)$.  Thus, the set~$A$ is also fixed point-wise. 

\vspace{2mm}
Now, we want to assign to every vertex $u\in A$ a suitable list $L(u)$ containing $c(\overrightarrow{vu})$ and  excluding $c(\overrightarrow{uv})$. Additionally, we have to ensure that adjacent vertices get different lists.  Clearly, $|N^L(A)|\leq \Delta,$ and the equality holds only if $A$ is an independent set of vertices of degree $\Delta$. For each $u\in A$, to establish its list $L(u)$, we choose $\lfloor\frac k2\rfloor-1$ colours from the set of $k-2$ colours since we require $c(\overrightarrow{vu})\in L(u)$ and $c(\overrightarrow{uv})\notin L(u)$. The list $L(u)$ has to be different from the lists assigned to vertices of $N^L(A)$ and from the list $L(w)$ of the root $w$. Hence, such a list $L(u)$ exists whenever
 $${k-2\choose \lfloor\frac k2\rfloor-1}\geq \Delta+1.$$
We claim that this inequality holds for each $\Delta\geq 13.$ 
As $k=\ceil[\Big]{2 \sqrt{\Delta}}$, it follows that $\Delta\leq \frac{k^2}{4}$. Thus, it suffices to show that the inequality
\begin{equation}\label{lists}
{k-2\choose \lfloor\frac k2\rfloor-1}\geq \frac{k^2}{4}+1
\end{equation}
holds for each $k\ge 8$. It clearly holds for $k=8$ and $k=9$. For greater values of $k$, we use induction taking into account the parity of $k$. Consider first the case when $k$ is even, and suppose that the inequality is true for some even $k\geq 8.$ We want to prove that it also holds for $k+2$. To this end observe that 
$${(k+2)-2\choose \frac{k+2}{2}-1}= {k-2\choose \frac k2-1}\frac {k(k-1)}{\frac{k^2}{4}}\ge \left(\frac{k^2}4+1\right)\frac {k(k-1)}{\frac{k^2}{4}}\geq k(k-1),$$
where the first inequality follows from the induction hypothesis. 
Thus, it suffices to show that $k(k-1)\ge \frac{(k+2)^2}{4}+1.$ This inequality is equivalent to the quadratic inequality $3k^2-8k-8\ge 0,$ which holds for $k\ge 4$. Analogously, we use induction for odd $k$.

Thus, we can assign lists $L(u)$ for all $u\in A$ when $\Delta\ge 13$ since $\lceil2\sqrt{13}\rceil=8$. Next, we colour all arcs incident to vertices of $A$ according to the rule $(\star)$, and continue the procedure for another set $A$ if it exists for the vertex $u$. Next, we consider a subsequent vertex~$v$ in the BFS ordering that has a child without an assigned list of colours, and repeat the procedure. If such a vertex does not exist, i.e. all vertices have assigned lists, then the obtained arc-colouring is proper and distinguishing, as desired. Consequently, our theorem holds for every graph with maximum degree $\Delta\ge 13$. 

\vspace{7mm}
Suppose now that $\Delta\le 12$, and there is a vertex $u\in A$ such that $c(\overrightarrow{vu})=\alpha$ and $c(\overrightarrow{uv})=\beta$ and all lists containing $\alpha$, but not $\beta$, are already assigned to the vertices of $N^L(A)\cup\{w\}$. There are $\lfloor\frac k2\rfloor(k-\lceil\frac k2\rceil)\geq \Delta$ distinct pairs of colours for pairs of opposite arcs between the parent $v$ with a given list $L(v)$ and its children.
At least one of them, $(\alpha_1,\beta_1)$, is not used, so we can recolour the arcs between $u$ and $v$ setting $c(\overrightarrow{vu})=\alpha_1$ and $c(\overrightarrow{uv})=\beta_1$. The number of lists $L$ such that $\{\alpha,\beta\}\setminus L=\{\beta\}$ is equal to ${k-2\choose \lfloor\frac k2\rfloor-1}$. The same is the number of lists $L$ with $\{\alpha_1,\beta_1\}\setminus L=\{\beta_1\}$. By the inclusion-exclusion principle, the number $m$ of lists $L$ such that $\{\alpha,\beta\}\setminus L=\{\beta\}$ and $\{\alpha_1,\beta_1\}\setminus L=\{\beta_1\}$ equals
$$m=2{k-2\choose \lfloor\frac k2\rfloor-1}-m_1,$$  where $m_1$ is the number of lists $L$ counted twice, that is, 
$$ m_1=\left\{\begin{array}{cll} {k-3\choose \lfloor\frac k2\rfloor-1} & {\rm if} & \{\alpha,\beta\}\cap \{\alpha_1,\beta_1\}=\{\alpha\}, \\ 
{k-3\choose \lfloor\frac k2\rfloor-2} & {\rm if} & \{\alpha,\beta\}\cap \{\alpha_1,\beta_1\}=\{\beta\}, \\ 
{k-4\choose \lfloor\frac k2\rfloor-2} & {\rm if} & \{\alpha,\beta\}\cap \{\alpha_1,\beta_1\}=\{\alpha,\beta\}, \\ 0 & {\rm if} & \{\alpha,\beta\}\cap \{\alpha_1,\beta_1\}=\emptyset. \end{array}
\right.$$
It is easy to verify that $m\ge \Delta+1$ for $\Delta\in\{7,8,10,11,12\}.$ Moreover, it follows that if all lists $L$ with $\{\alpha,\beta\}\setminus L=\{\beta\}$ are already assigned to the vertices of $N^L(A)\cup\{w\}$, then for any pair of colours $(\alpha_1,\beta_1)\neq (\alpha,\beta)$, there is a free list $L$ with $\{\alpha_1, \beta_1\}\setminus L=\{\beta_1\}.$ Therefore, we can assign a suitable list $L(u)$ to every vertex $u$ of $A$, for these values of $\Delta$.

\vspace{4mm}
For $\Delta=9$, we have $k=6$ colours. Since $m\ge 9$, we can find a suitable list $L(u)$ for each vertex $u\in A$ whenever $|N^L(A)|\le 8$. Suppose that $|N^L(A)|=9$ and that $L(w)=\{4,5,6\}$ is the only list available for a certain vertex $u\in A$ with $c(\overrightarrow{vu})=\alpha$ and $c(\overrightarrow{uv})=\beta$. It follows that $\alpha,\alpha_1\in\{4,5,6\}$ and $\beta,\beta_1\in\{1,2,3\}$. Moreover, the pairs $\{\alpha,\beta\}$ and $\{\alpha_1,\beta_1\}$ cannot be disjoint since otherwise we would have 10 free lists for $u$. Suppose first that $\alpha=\alpha_1$. Without loss of generality, we may assume that $\alpha=4,\beta=1$ and $\beta_1=2$. Thus, the eight lists containing 4 and not containing 1 or 2, different from $L(w)$, are assigned to the vertices of $N^L(A)$. In particular, $L(v_1)=\{1,3,4\}$ and $L(v_2)=\{2,3,4\}$ for some $v_1,v_2\in N^L(A)$. We put $L(u)=L(w)$ and colour the pairs of arcs between $u$ and $v_i$ with the same pair of colours $c(\overrightarrow{uv_i})=3, c(\overrightarrow{v_iu})=5$, for $i=1,2$ (recall that each vertex of $A\cup N^L(A)$ is already fixed). Consequently, the vertex $u$ has a colouring of incident arcs different from that of the root $w$, therefore the root $w$ is still fixed. If $\beta=\beta_1$, then we proceed in a similar way.

\vspace{3mm}
For $\Delta=6$, we have $k=5$ colours, and ten 2-element lists. If there is no free list $L(u)$ for some vertex $u\in A$, then we cancel the colouring of arcs between the vertex $v$ and its children. Clearly, $|A|\le 5$. Since $|N^L(A)|\le 6$, there are at least three lists $L_1$, $L_2$, and $L_3$ which are distinct from $L(w)$, which are not assigned to any vertex of $N(A)$ and hence can be used for the vertices of $A$. It is not difficult to verify that for two of them, say $L_1, L_2$, there are lists $L(v_1), L(v_2)$ assigned to some vertices $v_1,v_2\in N^L(A)$ such that $L_i\cap L(v_i)=\emptyset$ for $i=1,2$. If the independence number of $G[A]$ is at least 3, we assign $L_1$ to three independent vertices of $A$. We use distinct pairs of colours for pairs of opposite arcs between $v_1$ and the vertices with $L_1$ (there are four possible such pairs). The other vertices of $A$ get the list $L_2$, and we can colour the arcs between them and $v_1$ with pairs of colours, possibly equal, different from the pairs of colours of arcs between $v_1$ and vertices of $A$ with the list $L_1$. Next, we colour the arcs between $v_2$ and the vertices with $L_2$ with two distinct pairs of colours. As the vertices $v_1$ and $v_2$ are fixed, all the vertices of $A$ are fixed. We colour the remaining arcs of $\fsG[A\cup N^L(A)]$ according to the rule $(\star)$. 

Suppose that the independence number of $G[A]$ is less than 3. Hence $G[A]$ is a clique or $G[A]$ is spanned by two disjoint cliques of orders $n_1$ and $n_2$, where $1\le n_1\le n_2\le 4$. If these two cliques are isomorphic, then $n_1=n_2=2$ and $|A|=4.$ Hence $|N^L(A)|\le 5$, so there are four free lists. We assign them to the vertices of $A$ and colour the arcs of $\fsG[A\cup N^L(A)]$ in such a way that colours of the arcs between the parent $v$ and the two cliques are different, thus fixing all vertices of $A$. If $n_1< n_2\leq4$ or $A$ induces a clique of order $n_2$, then $|N^L(A)|\le 5-n_2+1\leq 4$ and we have at least five free lists. Finally, we colour all arcs of $\fsG[A\cup N^L(A)]$ following the rule ($\star$).

\vspace{2mm}
As we have shown that there exists a proper distinguishing colouring of $\fsG$ for $\Delta = 6$, it also always exists for $\Delta = 5$, because we have the same number $k=5$ of colours, and there are no more vertices in $A$ and $N^L(A)$. 

\vspace{3mm}

For $\Delta=4$, we have $k=4$ colours, and six 2-element lists. Clearly, $|A|\le 3$. Suppose that for some vertex $u\in A$, there is no free list compatible with colours of the arcs between $u$ and its parent $v$. In such a case, we remove colours between $v$ and all its children in $A$. 

If $A\cup N^L(A)$ has at most five vertices, then we can assign a distinct list $L(u)\neq L(w)$ to every vertex $u\in A$, and colour the arcs in $\fsG[A\cup N^L(A)]$ according to rule $(\star)$ so that each vertex of $A$ is fixed. The same holds if the number of distinct lists assigned to vertices in $N^L(A)$ is at most two. If there is an edge joining two vertices of $A$ (there may exist only one such edge), then $|N^L(A)|\le 3$, and two lists suffice for $A$, even if $|A|=3$. 

\vspace{4mm}


\begin{figure}[htb]
\centering

\includegraphics[width=8cm]{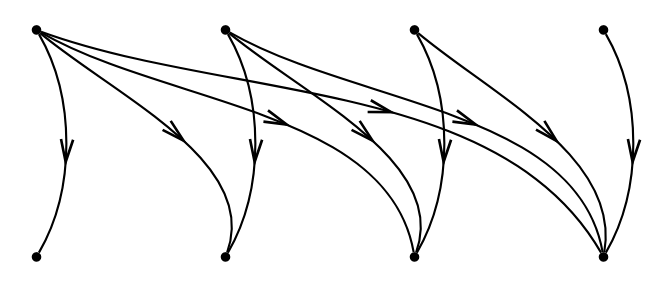}

\caption{Proper distinguishing colouring of $\overleftrightarrow{K_{4,4}}$ with three colours: only the arcs coloured with the third colour are drawn, all remaining arcs directed downwards are coloured with the first colour, and all arcs upwards with the second colour}\label{K44}
\end{figure}


Then suppose that $|A\cup N^L(A)|\geq 6$ and $A$ is an independent set. If $N^L(A)$ has four vertices with a common parent, then it must be the root $w$.  If $|A|=3$, then $G=K_{4,4}$ and a distinguishing proper colouring with three colours is presented in Figure~\ref{K44}. If $|A|=2$, then the subgraph $G[\{w\}\cup N^L(A)\cup A]$  is isomorphic to $K_{3,4}$, possibly with an edge or two independent edges. It is easy to find a proper distinguishing colouring of this subgraph using the fourth colour on one arc of each additional edge. 

Suppose now that there are at least two vertices of $N^L(A)$ with distinct parents.
If there is a vertex $v'\in N^L(A)$ such that $\overline{L(v')}$ is not a list of $w$ and of any other vertex in $N^L(A)$, then we assign the list $\overline{L(v')}$ to every vertex of $A$, and colour the arcs between them and $v'$ differently, thus fixing all vertices of $A$. Next, we properly colour the remaining arcs between $A$ and $N^L(A)$.

Hence, we are left with the situation, when 
\begin {enumerate}
\item $|A\cup N^L(A)|\geq 6$,
\item the vertices of $N^L(A)$ have at least three distinct lists,
\item every parent of a vertex in $N^L(A)$ has at most three children, 
\item for each vertex of $N^L(A)$ with a list $L$, the complement $\overline{L}$ is assigned to another vertex of $N^L(A)$ or  $\overline{L}=L(w)$.
\end{enumerate}
It follows that there exists a vertex $v_1\in N^L(A)$ such that $L(v_1)\ne \overline{L(w)}$ is unique in $N^L(A)$. Let $v'$ be a vertex of $N^L(A)$ with $L(v')=\overline{L(v_1)}$. We will change the list $L(v_1)$ so that $\overline{L(v')}$ will not be a list of any vertex in $N^L(A)\cup \{w\}$. 
To this end, first observe that for any pair of colours $(c(\overrightarrow{v_1x}),c(\overrightarrow{xv_1}))$, where $x$ is a parent of $v_1\in N^L(A)$, there are two possible lists for $v_1$. If the other list is different from $L(w)$, then we put the list $L(v')$ on the vertices of $A$ and proceed as above. Otherwise, given a list $L(x)$ we have four distinct pairs of colours between $x$ and its children (their number is at most three). Hence, we can change the pair of colours between $x$ and $v_1$ to obtain another possible list for $v_1$. This can be clearly done when $x$ is the only parent of $v_1$. 

If $v_1$ has another parent $x'$, then $|A|=2$ and $v_1$ has no children outside $A$. Hence, the four vertices in $N^L(A)$ have distinct lists that are pairwise complementary, because otherwise, we will have two free lists for vertices in $A$. Thus $\overline{L(w)}$ cannot be a list of any vertex in $N^L(A)$. It is easy to check that for any pair of lists $L(x),L(x')$ and a common child of $x,x'$, there are at least four distinct possible colourings of arcs on the path $xv_1x'$ of length two, when $L(x)\cap L(x')=\emptyset$, and even more when $L(x)\cap L(x')\ne\emptyset$. The number of common children of $x$ and $x'$ is at most three, so we can change the colours of arcs between $v_1$ and its two parents $x,x'$ to get a new list $L(v_1)$. If this new list is different from $L(w)$, then we are done. Otherwise, we let $L(v_1)=L(w)$, and colour the arcs between $v_1$ and $u_1,u_2\in A$ with distinct pairs of colours putting $(c(\overrightarrow{v_1u_1}),c(\overrightarrow{u_1v_1}))=
(c(\overrightarrow{v_1x}),c(\overrightarrow{xv_1}))$, thus fixing the vertices of $A$. We assign lists for the vertices $u_1$ and $u_2$.
Then we colour the other arcs between $A$ and $N^L(A)$ properly. Note that the root $w$ is still fixed since $w$ and~$v_1$ have different colourings of incident arcs.

\vspace{2mm}
As there always exists a proper distinguishing colouring of $\fsG$ for $\Delta = 4$, it also exists for $\Delta = 3$, because we have the same number $k=4$ of colours, and the numbers of vertices in sets $A$ and $N^L(A)$  are not greater. This completes the proof of Theorem~\ref{thm:main}.


\section{Forbidden monochromatic 2-paths only}\label{Slast}

In this section, we study proper arc-colourings, where only monochromatic 2-paths are forbidden.  To our knowledge, such colourings have not been investigated in the literature yet. Therefore, we first consider the chromatic index $\chi'_{2}(\overleftrightarrow{G})$. Clearly, for every symmetric digraph $\overleftrightarrow{G}$ we have
\begin{equation}\label{212}
\chi'_{2}(\overleftrightarrow{G}) \leq \chi'_{1,2}(\overleftrightarrow{G}).
\end{equation}

First, observe that the equality holds for bipartite symmetric digraphs of order at least 3.

\begin{observation}
For all connected bipartite symmetric digraphs $\overleftrightarrow{G}$, except for $\overleftrightarrow{K_2}$, it  holds $$ \chi'_{2}(\overleftrightarrow{G}) = \chi'_{1,2}(\overleftrightarrow{G}). $$
\end{observation}

\begin{proof}
By Theorem \ref{PRodl}, $\chi'_{1,2}(\overleftrightarrow{G})=2$ for any bipartite graph $G$. If $\overleftrightarrow{G}$ is a~connected symmetric digraph with $\chi'_{2}(\overleftrightarrow{G})=1$, then $\Delta(G)=1$ because monochromatic 2-paths are forbidden.
Consequently, either $\chi'_{2}(\overleftrightarrow{G})=2$  or $G \cong K_2$ if $\overleftrightarrow{G}$ is connected.
\end{proof}

On the other hand, let us show now that infinitely many graphs fulfill a strict inequality in~(\ref{212}).
\begin{proposition}\label{pelne2}
Let $l\geq 2$ be any integer and let $n={2l+1\choose l}+1$. Then   $$\chi'_2(\overleftrightarrow{K_n})<\chi'_{1,2}(\overleftrightarrow{K_n}).$$
\end{proposition}

\begin{proof}
Let $n={{2l+1}\choose {l}}+1$ for $l\geq 2$. It follows from Theorem~\ref{PRodl} that $\chi'_{1,2}(\overleftrightarrow{K_n})=2l+2$.

We divide the set of vertices of $\overleftrightarrow{K_n}$ into pairs $M_i=\{u_i,v_i\}, i=1,\ldots,\lfloor \frac n2\rfloor$, and, possibly, a~single vertex $M_{\lceil \frac n2\rceil}=\{u_n\}$ if $n$ is odd. For each $i=1,\ldots, \lfloor \frac n2\rfloor$, we colour the arcs $\overrightarrow{u_iv_i},\overrightarrow{v_iu_i}$ with one and the same colour, which will no longer be used. Next, consider a complete symmetric digraph $\overleftrightarrow{K_{\lceil \frac n2\rceil}}$ with vertices $M_1,\ldots,M_{\lceil \frac n2\rceil}$, and its proper arc-colouring $c$ of type I with $\chi'_{1,2}(\overleftrightarrow{K_{\lceil \frac n2\rceil}})$ colours, i.e. without monochromatic 2-cycles and 2-paths. In the digraph $\overleftrightarrow{K_n}$, for $i\neq j$, we colour all arcs between the vertices of $M_i$ and $M_j$ with the colour $c(\overrightarrow{M_iM_j})$. It is easy to see that we thus obtain an arc-colouring of $\overleftrightarrow{K_n}$ with $1+\chi'_{1,2}(\overleftrightarrow{K_{\lceil \frac n2\rceil}})$ colours  not creating any monochromatic 2-path. Hence, $\chi'_2(\overleftrightarrow{K_n})\leq 1+\chi'_{1,2}(\overleftrightarrow{K_{\lceil \frac n2\rceil}})$.

Denote $a_k={k\choose \lfloor\frac k2\rfloor}$ for $k\geq 3$, and observe that $n-1={2l+1\choose l}=a_{2l+1}$. Moreover, $a_k\geq k$ and
$$\frac{a_{2l+1}}{a_{2l}}= \frac{2l+1}{l+1},$$
for each $k,l$. Consequently,
$$a_{2l}=a_{2l+1}\cdot\frac{l+1}{2l+1}=(n-1)\cdot\frac{l+1}{2l+1}=\frac n2+\frac {n}{2(2l+1)}-\frac{l+1}{2l+1}=$$
$$=\frac n2+\frac{a_{2l+1}-(2l+1)}{2(2l+1)}.$$
Hence, $a_{2l}\geq \lceil\frac n2\rceil$ since $a_{2l+1}\geq 2l+1$. Therefore, $\chi'_{1,2}(\overleftrightarrow{K_{\lceil \frac n2\rceil}})\leq 2l$, and
$$\chi'_2(\overleftrightarrow{K_n})\leq 2l+1<2l+2=\chi'_{1,2}(\overleftrightarrow{K_n}).$$
\end{proof}

\vspace{2mm}

Now, we investigate the distinguishing chromatic index $\chi'_{D_2}(\overleftrightarrow{G})$.
Clearly, for every symmetric digraph $\overleftrightarrow{G}$ it holds
\begin{equation}\label{D2D12}
\chi'_{D_2}(\overleftrightarrow{G}) \leq \chi'_{D_{1,2}}(\overleftrightarrow{G}).
\end{equation}
Hence, Theorem~\ref{thm:main} immediately implies the following.
\begin{proposition}
If $G$ is a connected graph, then 
$$\chi'_{D_2}(\overleftrightarrow{G}) \leq \ceil[\Big]{2 \sqrt{\Delta}}.$$
\end{proposition} 

This bound is optimal because the equality holds for every symmetric directed star $\overleftrightarrow{K_{1,\Delta}}$ with $\Delta \geq 4$, as shown below.

\begin{observation}
If $\Delta\ge 4$, then $\chi'_{D_2}(\overleftrightarrow{K_{1,\Delta}})=\lceil 2\sqrt{\Delta}\rceil$
\end{observation}
\begin{proof}
Let $w$ be a central vertex of the star $K_{1,\Delta}$.  Consider a proper distinguishing arc-colouring $c$ of $\overleftrightarrow{K_{1,\Delta}}$ with $\chi'_{D_2}(\overleftrightarrow{K_{1,\Delta}})$ colours, without monochromatic 2-paths, and with the least number of monochromatic 2-cycles. If every 2-cycle is monochromatic, then $c$ uses $\Delta$ colours, and $\Delta\geq \lceil 2\sqrt{\Delta}\rceil$ except for $\Delta=3$. Otherwise, there exists a pendant vertex $v$ with $c(\overrightarrow{wv}) \neq c(\overrightarrow{vw})$. If there is a~pendant vertex $u$ with $c(\overrightarrow{wu}) = c(\overrightarrow{uw})$, then we can recolour one of the arcs $\overrightarrow{wu}$ or $\overrightarrow{uw}$ by a colour of $\overrightarrow{wv}$ or $\overrightarrow{vw}$, thus obtaining a distinguishing colouring with the same number of colours, a smaller number of monochromatic 2-cycles and without monochromatic 2-paths. In this way, we can obtain a proper distinguishing arc-colouring with the same number $\chi'_{D_2}(\overleftrightarrow{K_{1,\Delta}})$ of colours and without monochromatic 2-cycles. Hence, $\chi'_{D_2}(\overleftrightarrow{K_{1,\Delta}})=\chi'_{D_{1,2}}(\overleftrightarrow{K_{1,\Delta}})=\lceil 2\sqrt{\Delta}\rceil$, by Proposition~\ref{stars}.
\end{proof}

On the other hand, there exist infinitely many symmetric digraphs $\overleftrightarrow{G}$ with $\chi'_{D_2}(\overleftrightarrow{G}) < \chi'_{D_{1,2}}(\overleftrightarrow{G})$. Simple examples are paths of odd order. Indeed, colouring every second 2-cycle of $\overleftrightarrow{P_{2k+1}}$ with the same colour yields a distinguishing proper arc-colouring. Thus, $\chi'_{D_2}(\overleftrightarrow{P_{2k+1}})=2$ while $\chi'_{D_{1,2}}(\overleftrightarrow{P_{2k+1}})=3$ by Observation~\ref{cycles}. 

Besides, infinitely many complete symmetric digraphs fulfill the strict inequality. 

\begin{proposition}
Let $l\geq 2$ be any integer and let $n={2l+1\choose l}+1$. Then  $$\chi'_{D_2}(\overleftrightarrow{K_n})<\chi'_{D_{1,2}}(\overleftrightarrow{K_n}).$$
\end{proposition}
\begin{proof}
In view of Proposition \ref{chi12} and the proof of Proposition~\ref{pelne2}, it suffices to show that 
$$ \chi'_{D_2}(\overleftrightarrow{K_n}) \leq   \chi'_{1,2}(\overleftrightarrow{K_{\ceil{\frac{n}{2}}}}) + 1$$
for every $n={2l+1\choose l}+1$.

As in the proof of Proposition~\ref{pelne2}, we partition the vertex set of $\overleftrightarrow{K_n}$ into pairs $M_i=\{u_i,v_i\}, i=1,\ldots,\lfloor \frac n2\rfloor$, and, possibly, a single vertex $M_{\lceil \frac n2\rceil}=\{u_n\}$ if $n$ is odd. For each $i=1,\ldots, \lfloor \frac n2\rfloor$, we colour the arcs $\overrightarrow{u_iv_i},\overrightarrow{v_iu_i}$ with one and the same colour, which will no longer be used. Next, we consider a complete symmetric digraph $\overleftrightarrow{K_{\lceil \frac n2\rceil}}$ with vertices $M_1,\ldots,M_{\lceil \frac n2\rceil}$, and produce its arc-colouring $c$ without monochromatic 2-cycles and 2-paths as follows. Let $k=\chi'_{1,2}(\overleftrightarrow{K_{\lceil \frac n2\rceil}})$. That is,
$$k=\min \left\{k'\;:\; \frac n2\leq {k'\choose \lfloor\frac{k'}{2}\rfloor}\right\}.$$
Each vertex $M_i$ gets a distinct list $L(M_i)$ of $\lfloor \frac k2\rfloor$ colours from a fixed set of $k$ colours which are admissible for arcs ingoing to $M_i$. We colour the arc $\overrightarrow{M_iM_j}$ with any colour from $L(M_j)\setminus L(M_i)$. By Proposition~\ref{chi12}, this is a distinguishing colouring of $\overleftrightarrow{K_{\lceil \frac n2\rceil}}$. Then we obtain an arc-colouring of $K_n$ without monochromatic 2-paths by colouring each arc from $M_i$ to $M_j$ with $c(\overrightarrow{M_iM_j})$ for $i\neq j$.

Clearly, this colouring of $\overleftrightarrow{K_n}$ is not distinguishing but it suffices to recolour some arcs to break all transpositions of the vertices $u_i,v_i$, for $i=1,\ldots,\lfloor\frac n2\rfloor$. To do this, for each such $i$, we choose a $j$ such that $L(M_i)\setminus(L(M_j)\cup \{c(\overrightarrow{M_jM_i}\})$ contains a colour, say $\alpha$. We recolour the arc $u_ju_i$ with $\alpha$. Such a $j$ exists for each $i$, by the definition of $k$ since $\frac k2\geq 2$ for $n\geq 7$. This new arc-colouring of $\overleftrightarrow{K_n}$ still has no monochromatic 2-paths, and it is distinguishing because each vertex $v_i$ is fixed. Indeed, if $\varphi$ is an automorphism of $\overleftrightarrow{K_n}$ preserving this colouring, then $v_i$ cannot be mapped by $\varphi$ onto $u_i$ since the multisets of arcs ingoing to $u_i$ and to $v_i$ are distinct. Also, $v_i$ cannot be mapped onto any other vertex because there are no monochromatic 2-paths.
\end{proof}

We conclude this section with a conjecture.

\begin{conjecture}
$$ \chi'_{D_{2}}(\overleftrightarrow{K_{n}}) = \chi'_{2}(\overleftrightarrow{K_{\ceil{\frac{n}{2}}}})+1. $$
\end{conjecture}


\end{document}